\documentclass[12pt,a4paper]{article}

\usepackage[utf8]{inputenc}
\usepackage[english]{babel}
\usepackage{amsmath,verbatim,amsfonts,amsthm,amssymb,mathrsfs,stmaryrd}
\usepackage{geometry}
\usepackage{graphicx}
\immediate\write18{texcount -tex -sum  \jobname.tex > \jobname.wordcount.tex}
\usepackage{stmaryrd}
\SetSymbolFont{stmry}{bold}{U}{stmry}{m}{n}

\usepackage{ifthen}

\usepackage[bookmarksnumbered]{hyperref}

\usepackage{eucal}	
\usepackage{tensor}
\usepackage[dvipsnames]{xcolor}
\usepackage{csquotes}
\usepackage{tabularx}
\usepackage{blindtext}
\usepackage{multicol}
\usepackage{enumitem}
\usepackage[normalem]{ulem} 
\usepackage{tikz}  
\usetikzlibrary{arrows,positioning}
\usepackage{indentfirst}

\usepackage[all]{xy}
\usepackage{amsfonts} 
\usepackage{textcomp}
\usepackage{color}
\graphicspath{ {./images/} }

\usepackage[noabbrev, capitalise, nameinlink, nosort]{cleveref}
\crefname{enumi}{item}{items}

\usepackage[T1]{fontenc}
\usepackage{calligra}

\newcounter{generalnumbering}   \numberwithin{generalnumbering}{section}

\theoremstyle{plain} \newtheorem{theorem}[generalnumbering]{Theorem}
\theoremstyle{plain} \newtheorem{corollary}[generalnumbering]{Corollary}
\theoremstyle{definition} \newtheorem{definition}[generalnumbering]{Definition}
\theoremstyle{definition} \newtheorem{example}[generalnumbering]{Example}
\theoremstyle{plain} \newtheorem{proposition}[generalnumbering]{Proposition}
\theoremstyle{plain} \newtheorem{lemma}[generalnumbering]{Lemma}
\theoremstyle{definition} \newtheorem{remark}[generalnumbering]{Remark}
\theoremstyle{definition}

\newcommand{\namefordifferentenvironment}{}

\theoremstyle{plain}    \newtheorem{plainstyle}[generalnumbering]{\namefordifferentenvironment}
\theoremstyle{plain}    \newtheorem*{plainstyle*}{\namefordifferentenvironment}
\theoremstyle{definition}    \newtheorem{definitionstyle}[generalnumbering]{\namefordifferentenvironment}
\theoremstyle{definition}    \newtheorem*{definitionstyle*}{\namefordifferentenvironment}

\newenvironment{penv*}[1]{\renewcommand{\namefordifferentenvironment}{#1}\begin{plainstyle*}}{\end{plainstyle*}}

\newenvironment{denv*}[1]{\renewcommand{\namefordifferentenvironment}{#1}\begin{definitionstyle*}}{\end{definitionstyle*}}

\DeclareMathOperator{\Se}{\mathcal{S}}
\DeclareMathOperator{\Te}{\mathcal{T}}
\DeclareMathOperator{\Ge}{\mathcal{G}}
\DeclareMathOperator{\ix}{\mathcal{I}(X)}

\DeclareMathOperator{\id}{id}
\DeclareMathOperator{\dom}{dom}
\DeclareMathOperator{\ran}{ran}

\DeclareMathOperator{\orb}{Orb}

\newcommand*{\simrum}{\overset{\scriptscriptstyle{\ref{r1}}}{\sim}}
\newcommand*{\simrdois}{\overset{\scriptscriptstyle{\ref{r2}}}{\sim}}

\title{A $P$‑theorem for Inverse Semigroupoids through Ordered Globalizations}

\author{
\small{Paulinho Demeneghi}\\ 
\footnotesize{Universidade Federal de Santa Catarina, Florianópolis, Brazil}\\
\footnotesize{e-mail: \href{mailto:paulinho.demeneghi@ufsc.br}{paulinho.demeneghi@ufsc.br}}\\
\small{Willian Goulart Gomes Velasco}\\
\footnotesize{Universidade Tecnológica Federal do Paraná, Curitiba, Brazil}\\
\footnotesize{e-mail: \href{mailto:willianvelasco@protonmail.com}{willianvelasco@protonmail.com}}\\
\small{Víctor Marín}\\
\footnotesize{Universidad del Tolima, Ibagué, Colombia}\\
\footnotesize{e-mail:\href{mailto:vemarinc@ut.edu.co}{vemarinc@ut.edu.co}}\\
\small{Felipe Augusto Tasca}\\
\footnotesize{Instituto Federal do Paraná, União da Vitória, Brazil}\\
\footnotesize{e-mail: \href{mailto:tasca.felipe@gmail.com}{tasca.felipe@gmail.com}}
}

\date{}

\providecommand{\keywords}[1]
{
  \small	
  \textbf{\textit{}} #1
}    

\begin{document}

\maketitle

\begin{abstract}
We prove that every ordered partial action of an inverse semigroupoid on a partially ordered set admits a universal globalization. This result is used to establish a connection between ordered partial actions of groupoids and a multi-object analogue of McAlister triples. As a consequence, we obtain a multi-object version of the $P$-theorem: every $E$-unitary inverse semigroupoid is isomorphic to a semidirect product arising from an ordered partial action of a groupoid on a multi-object version of a semilattice.

\end{abstract}


\section{Introduction}

Since their introduction in the early 1950s, inverse semigroups have become one of the most extensively studied classes of semigroups. Among them, the subclass of $E$-unitary inverse semigroups, introduced by Saito in \cite{Saito} under the name \emph{proper inverse semigroups}, has played a particularly prominent role. Their significance is largely attributable to the foundational work of McAlister.

A landmark result in the theory is the structural theorem known as the $P$-theorem, proved by McAlister in \cite[Theorem 2.6]{McAlister2}. The central idea behind this theorem is that groups and semilattices, the simplest examples of $E$-unitary inverse semigroups, can be used as foundational components for building all $E$-unitary inverse semigroups.

A class of $E$-unitary inverse semigroup obtained in this way is that of $P$-semigroups. A $P$-semigroup is obtained from a McAlister triple $(G,E,X)$, which consists of a group $G$ acting by order automorphisms on a partially ordered set $E$, an order ideal $X$ of $E$ that is a meet semilattice under the induced order and the action is required to satisfy the conditions $G \cdot X = E$ and $g \cdot X \cap X \neq \emptyset$ for all $g \in G$. The $P$-semigroup associated with $(G,E,X)$ is then formed by $P(G,E,X)=\{(g,x)\in G\times X \mid g \cdot x \in X\}$ with product given by $(g,x)\cdot (h,y)=(gh,h^{-1}\cdot x \land y)$. The McAlister $P$-theorem asserts that every $E$-unitary inverse semigroup is isomorphic to a $P$-semigroup.

A different perspective was explored by Kellendonk and Lawson in \cite{KL2004}, who, building on Munn's earlier work \cite{Munn1976}, framed the $P$-theorem in terms of partial actions of groups on partially ordered sets. They observed that, in a McAlister triple $(G,E,X)$, the restriction of the global action of $G$ on the poset $E$ to the order ideal $X$ defines a partial action of $G$ on $X$ by order isomorphisms, and that the global action corresponds essentially to a globalization of this partial action. Then, they proved that every partial action of a group on a partially ordered set by order isomorphisms admits a globalization, and used this to construct McAlister triples as follows: given a partial action of a group $G$ on a semilattice $X$, they produced a globalization through an action of $G$ on a partially ordered set $E$ by order automorphisms, along with an order ideal $Y$ of $E$ such that $Y$ is order-isomorphic to $X$ on which the induced partial action of $G$ on $Y$ is equivalent to the original partial action. The resulting triple $(G,E,Y)$ is a McAlister triple.

One main goal of this paper is to establish a multi-object analogue of the $P$-theorem for $E$-unitary inverse semigroupoids. Such structures have appeared, for instance, in \cite{liu,steinberg_2003_lifting}  and a multi-object analogue of a McAlister triple may naturally be defined as a triple $(\Ge,E,X)$ in which $\Ge$ is a groupoid, $E$ is a partially ordered set and an order ideal $X$ of $E$ which is a semilatticeoid (a multi-object version of a semilattice) such that $\Ge$ acts on $E$ and the action satisfies $\Ge\cdot X = E$ and $g \cdot X \cap X \neq \emptyset$ for all $g\in\Ge$.

The framework of inverse semigroupoids offers a unified setting for treating both inverse semigroups and groupoids. The theory of inverse semigroupoids has seen increasing attention in recent years, with contributions such as \cite{velasco-muniz,Luiz2,Luiz1,demeneghi-tasca-2024,dewolf-ehresmann,thaisa2020,liu}. Informally, a semigroupoid can be regarded as a set equipped with a partially defined associative operation. Two distinct notions of semigroupoids appear in the literature: Tilson's, introduced in \cite{tilson}, and Exel's, presented in \cite{exelsemigroupoid1}. Although Exel's definition is more general, we adopt Tilson's due to its graph-based composition framework. Loosely speaking, a Tilson semigroupoid is a small category that may lack identity arrows. Ultimately, when the pseudo-inverse axiom is imposed in both frameworks, they give rise to equivalent notions, as proved by Cordeiro in \cite{Luiz1}.

To establish a multi-object analogue of the $P$-theorem for $E$-unitary inverse semigroupoids, we follow the approach developed by Kellendonk and Lawson \cite{KL2004}. This requires addressing the problem of the existence of globalizations for partial actions of groupoids on partially ordered sets. In this work, however, we adopt a more general perspective by considering the globalization problem for partial actions of inverse semigroupoids on partially ordered sets.

Partial actions of inverse semigroupoids on sets were introduced by Cordeiro in \cite{Luiz1}. In \cite{demeneghi-tasca-2024}, Demeneghi and Tasca proved that every partial action of an inverse semigroupoid on a set admits a universal globalization. We extend this result by introducing the notion of ordered partial actions of inverse semigroupoids on partially ordered sets (\cref{def:ordered_partial_action}) and prove that every ordered partial action of an inverse semigroupoid on a poset admits a globalization (\cref{o_teorema_1}). Furthermore, this globalization is universal among all global actions that extend the original partial action in an appropriate sense (\cref{prop:universal_globalization}).

We apply this result to establish a connection between generalized McAlister triples and ordered partial actions of groupoids on semilatticeoids (\cref{TMP}). This leads to the definition of a $P$-semigroupoid as an inverse semigroupoid isomorphic to a semidirect product arising from such a partial action. In particular, $P$-semigroupoids are shown to be $E$-unitary (\cref{prop:P-semigroup_is_E-unitary}).

Finally, by adapting Munn’s ideas to the multi-object setting, we establish a version of the $P$-theorem in this context (\cref{P-theoremConstruct2}). Specifically, we show that every $E$-unitary inverse semigroupoid $\Se$ is isomorphic to the semidirect product arising from the ordered partial action of its maximal groupoid image $\Se/\sigma$ on $E(\Se)$, induced by the Munn action of $\Se$ on $E(\Se)$ (\cref{Munn-global}).

Summarizing, the paper is structured as follows. For the convenience of the reader, we establish all the necessary preliminaries on inverse semigroupoids with particular focus on $E$-unitary inverse semigroupoids, in \cref{sec2}. In \cref{sec3}, we develop the theory of ordered actions of inverse semigroupoids on partially ordered sets an address the globalization problem for such actions. In \cref{P-theorem}, we explore the connection between the globalization problem for ordered groupoid actions and McAlister triples, culminating in a multi-object version of the $P$-theorem.


\section{Preliminaries on Inverse Semigroupoids}\label{sec2}

This opening section outlines the essential theory of inverse semigroupoids. For further reference, we recommend \cite{Luiz1, exelsemigroupoid2, exelsemigroupoid1, liu, tilson}.

In this paper, we employ the graphed semigroupoid definition - sometimes termed a semicategory - according to Tilson \cite{tilson}. By a graph, we mean a quadruple $(\Se,\Se^{(0)},d,c)$, where $\Se$ and $\Se^{(0)}$ are sets called the \emph{arrow} and \emph{object} sets, respectively, and $d,c: \Se \to \Se^{(0)}$ are maps referred to as the \emph{domain} and \emph{codomain} maps, respectively. We denote by $\Se^{(2)}\subseteq \Se \times \Se$ the set of all \emph{composable pairs}, that is, the set consisting of all pairs $(s,t) \in \Se \times \Se$ such that $d(s)=c(t)$.

\begin{definition}\label{def_semigroupoid}
A \emph{semigroupoid} is a quintuple $(\Se,\Se^{(0)},d,c,\mu)$, where $(\Se,\Se^{(0)},d,c)$ is a graph and $\mu: \Se^{(2)} \to \Se$ is a map, called the \emph{multiplication} map (denoted by juxtaposition), given by $\mu(s,t)=st$, satisfying the following conditions:
\begin{enumerate}[label=(\roman*)]
    \item\label{item_i} $d(st)=d(t)$ and $c(st)=c(s)$ for any composable pair $(s,t) \in \Se^{(2)}$;
    \item\label{item_ii} $(rs)t=r(st)$ whenever $(r,s), (s,t) \in \Se^{(2)}$.
\end{enumerate}
\end{definition}

Slightly abusing notation, we denote the semigroupoid $(\Se,\Se^{(0)},d,c,\mu)$ by $(\Se,\Se^{(0)})$, or simply by $\Se$. To emphasize the semigroupoid when needed, we employ subscripts, $d_{\Se}$ and $c_{\Se}$.

Notice that \cref{item_i} justifies \cref{item_ii} in the following sense: if $r,s,t \in \Se$ are elements such that $(r,s)$ and $(s,t)$ lie in $\Se^{(2)}$, then \cref{item_i} implies that $(rs,t)$ and $(r,st)$ also lie in $\Se^{(2)}$.

We assume throughout that $\Se^{(0)}=d(\Se)\cup c(\Se)$. This does not cause a loss of generality, as we can always replace $\Se^{(0)}$ by $d(\Se)\cup c(\Se)$, obtaining a new semigroupoid that satisfies the desired condition.

Semigroupoids encompass both semigroups and small categories.
In particular, a semigroup is a semigroupoid with exactly one object, and a small category is a semigroupoid endowed with an identity for each object.

To be precise, by an \emph{identity element} in a semigroupoid $\Se$, we mean an element $e \in \Se$ that satisfies $es=s$ and $te=t$ for every $s,t \in \Se$ whenever $(e,s), (t,e) \in \Se^{(2)}$. If necessary, we may say that $e \in \Se$ is an identity over an object $u \in \Se^{(0)}$ to emphasize that $e$ is an identity element such that $d(e)=c(e)=u$.

In a semigroupoid $\Se$, we call an element $e\in\Se$ \emph{idempotent} if $(e,e)\in\Se^{(2)}$ and $e^2=e$. The identity elements are then special cases of such idempotents. Once again, we may emphasize that $d(e)=c(e)=u$ by stating that $e$ is an idempotent over $u \in \Se^{(0)}$. We denote by $E(\Se)$ the set of idempotent elements of $\Se$.

We say that $s,t \in \Se$ are \emph{parallel} if $d(s)=d(t)$ and $c(s)=c(t)$.  

\begin{definition}
Let $\Se$ be a semigroupoid. A subset $\Te \subseteq \Se$ is said to be a \emph{subsemigroupoid} of $\Se$ if it is closed by multiplication.
\end{definition}

Notice that if $\Te$ is a subset of $\Se$ that is closed under multiplication, then setting $\Te^{(0)}=d(\Te)\cup c(\Te)$ and restricting the domain, codomain and multiplication maps appropriately, we obtain a semigroupoid $(\Te, \Te^{(0)},d,c,\mu)$.  

We focus on the class of semigroupoids known as inverse semigroupoids and, in particular, their subclass of groupoids.

\begin{definition}\label{def_inv_semigroupoid}
A semigroupoid $\Se$ is said to be an \emph{inverse semigroupoid} if, for every $s \in \Se$, there exists a unique element $s^* \in \Se$, called the \emph{inverse} of $s$, such that $(s,s^*), (s^*,s) \in \Se^{(2)}$ and  
$$ss^*s = s \quad \text{ and } \quad s^*ss^* = s^*.$$  
\end{definition}

In an inverse semigroupoid, we have $d(s^*)=c(s)$ and $c(s^*)=d(s)$ for every $s \in \Se$. Hence, the assumption $\Se^{(0)}=d(\Se)\cup c(\Se)$, which we previously imposed for general semigroupoids, becomes $\Se^{(0)}=d(\Se)=c(\Se)$ for inverse semigroupoids. This condition will be assumed throughout the remainder of this paper.

Similarly, just as semigroupoids generalize semigroups and small categories, the class of inverse semigroupoids encompasses inverse semigroups, groupoids, and inverse categories. In fact, inverse semigroups correspond to inverse semigroupoids with a single object, while groupoids are precisely those inverse semigroupoids that admit a unique idempotent over any given object $v \in \Se^{(0)}$. Moreover, inverse categories are inverse semigroupoids in which every object has an identity.

As our interest also lies in groupoids, we shall use $\Ge$ to denote an arbitrary groupoid. We identify each object $u\in\Ge^{(0)}$ with the identity over $u$. With this identification, we regard $\Ge^{(0)}$ as a subset $E(\Ge)$ of $\Ge$ and write $d(g)=g^{-1}g$ and $c(g)=gg^{-1}$ for every $g\in\Ge$.

In the framework of inverse semigroupoids, many standard results from the theory of inverse semigroups can be generalized with minor adjustments. For instance, the set $E(\Se)$ of idempotent elements commutes in the sense that if $e,f \in E(\Se)$ and $(e,f) \in \Se^{(2)}$, then $(f,e) \in \Se^{(2)}$ and $ef = fe$ (see Lemma 3.3.1 of \cite{liu}).

We now summarize the fundamental rules governing the involution on $\Se$: $(s^*)^* = s$ for all $s \in \Se$, and if $(s,t) \in \Se^{(2)}$, then $(t^*,s^*) \in \Se^{(2)}$ and $(st)^* = t^*s^*$. Moreover, for any $s \in \Se$, the elements $s^*s$ and $ss^*$ belong to $E(\Se)$, and for every $e \in E(\Se)$, we have $e^* = e$, which implies that every idempotent is of the form $s^*s$ for some $s \in \Se$. Finally, for any $s \in \Se$ and $e \in E(\Se)$ with $(e,s) \in \Se^{(2)}$, the element $s^*es$ belongs to $E(\Se)$. 

Since the idempotents in $E(\Se)$ commute,  we obtain a natural partial order on any inverse semigroupoid $\Se$. Specifically, for parallel arrows $s,t \in \Se$, we write $s \leqslant t$ if $s = te$ for some $e \in E(\Se)$ such that $(t,e) \in \Se^{(2)}$. According to Lemma 3.3.6 of \cite{liu},  $s \leqslant t $ if, and only if, $s = t\,s^*s$, or equivalent $s=ft$ for some $f \in E(\Se)$ with $(f,t) \in \Se^{(2)}$, or $s = ss^*t$.  

This order relation is called the \emph{natural partial order} on $\Se$. In particular, for any $s,t \in \Se$, we have $s \leqslant t$ if and only if $s^* \leqslant t^*$. Additionally, if $s_1, s_2, t_1, t_2 \in \Se$ satisfy $(s_1,s_2) \in \Se^{(2)}$, $s_1\leqslant t_1$, and $s_2\leqslant t_2$, then $(t_1,t_2) \in \Se^{(2)}$ and $s_1s_2\leqslant t_1t_2$. Furthermore, it follows directly from the definition that if $e,f \in E(\Se)$ and $(e,f)\in \Se^{(2)}$, then $ef\leqslant e,f$. Consequently, we obtain $s^*es\leqslant s^*s$ for all $s \in \Se$ and $e \in E(\Se)$ such that $(e,s) \in \Se^{(2)}$.  

\begin{remark}
Notice that an inverse semigroupoid $\Se$ is a groupoid if and only if the natural partial order coincides with the equality relation. Indeed, if $\Se$ is a groupoid, it is straightforward to verify that the natural partial order reduces to equality. Conversely, suppose that the natural partial order is the equality relation. If $e,f\in E(\Se)$ are idempotents over a given object $v \in \Se^{(0)}$, then $(e,f)\in\Se^{(2)}$ and $ef\leqslant e,f$, which implies $e=f$. Consequently, for each object $u \in \Se^{(0)}$, there exists a unique idempotent over $u$, and thus $\Se$ is a groupoid.
\end{remark}

Given two graphs $(\Se,\Se^{(0)},d_{\Se},c_{\Se})$ and $(\Te,\Te^{(0)},d_{\Te},c_{\Te})$, a morphism of graphs is a pair $(\varphi,\varphi^{(0)})$, where $\varphi\colon\Se\to\Te$ and $\varphi^{(0)}\colon\Se^{(0)}\to\Te^{(0)}$ are functions satisfying  
$$
d_{\Te}\circ\varphi=\varphi^{(0)}\circ d_{\Se} \quad \text{and} \quad c_{\Te}\circ\varphi=\varphi^{(0)}\circ c_{\Se}.
$$
In particular, a morphism of graphs  $(\varphi,\varphi^{(0)})$ preserves composability in the sense that $(\varphi(s),\varphi(t))\in\Te^{(2)}$ for every $(s,t)\in\Se^{(2)}$.

Conversely, if $d_{\Se}(\Se)=\Se^{(0)}=c_{\Se}(\Se)$ and $\varphi\colon\Se\to\Te$ is a function such that $(\varphi(s),\varphi(t))\in\Te^{(2)}$ for all $(s,t)\in\Se^{(2)}$, then there exists a unique function $\varphi^{(0)}\colon\Se^{(0)}\to\Te^{(0)}$ making $(\varphi,\varphi^{(0)})$ a morphism of graphs, as established in Proposition 2.19 of \cite{Luiz1}. This discussion motivates the following definition.

\begin{definition}\label{def_morphism}
Let $\Se$ and $\Te$ be semigroupoids. A map $\varphi\colon\Se\to\Te$ is said to be a \emph{morphism (of semigroupoids)} if $\big(\varphi(s),\varphi(t)\big)\in\Te^{(2)}$ and $\varphi(st)=\varphi(s)\varphi(t)$ for every pair $(s,t)\in\Se^{(2)}$.
\end{definition}

In the case where $\Se$ is an inverse semigroupoid, we assume that $d_{\Se}(\Se) = \Se^{(0)} = c_{\Se}(\Se)$. Thus, by the discussion preceding \cref{def_morphism}, for any morphism $\varphi\colon\Se\to\Te$, there exists a unique map $\varphi^{(0)}\colon\Se^{(0)}\to\Te^{(0)}$ such that $(\varphi,\varphi^{(0)})$ is a graph morphism. With this in mind, to emphasize the object map $\varphi^{(0)}$, we may sometimes refer to $(\varphi,\varphi^{(0)})$ as a morphism of semigroupoids.

Notice that if $\Se$ and $\Te$ are inverse semigroupoids and $\varphi\colon\Se\to\Te$ is a morphism of semigroupoids, then $\varphi$ necessarily preserves both inverses and the natural partial order. Specifically, for every $s\in\Se$, we have $\varphi(s^*)=\varphi(s)^*$, and for any parallel arrows $s,t\in\Se$ with $s\leqslant t$, it follows that $\varphi(s)\leqslant\varphi(t)$.

If $\varphi\colon\Se\to\Te$ is a morphism, one readily checks that $\varphi\big(E(\Se)\big)\subseteq E(\Te)$. In general, however, $\varphi$ may map non-idempotent elements of $\Se$ to idempotent elements of $\Te$.

\begin{definition}\label{idempotent_pure_map}
Let $\Se$ and $\Te$ be inverse semigroupoids. A semigroupoid morphism $\varphi\colon\Se\to\Te$ is said to be \emph{idempotent pure} if $\varphi^{-1}(E(\Te)) \subseteq E(\Se)$ (and hence, $\varphi^{-1}(E(\Te)) = E(\Se)$).
\end{definition}

An equivalence relation $R$ on graph $\Se$ is called \emph{graphed} if the domain and codomain maps are $R$-invariant, meaning that $d(s)=d(t)$ and $c(s)=c(t)$ for every pair $(s,t)\in R$.

\begin{definition}\cite[Definition 4.3]{Luiz1}
Let $\Se$ be a semigroupoid. A relation $R$ on $\Se$ is said to be a \emph{congruence} on $\Se$ if $R$ is a graphed equivalence relation on $\Se$ and $(s_1s_2, t_1t_2) \in R$ for all $(s_1, t_1), (s_2, t_2) \in R$ such that $(s_1, s_2), (t_1, t_2) \in \Se^{(2)}$. Two elements $s,t\in\Se$ are said to be $R$-\emph{congruent} if $(s,t)\in R$.
\end{definition}

Observe that the above definition requires only $(t_1, t_2) \in \Se^{(2)}$ or $(s_1, s_2) \in \Se^{(2)}$ since $d(s_1) = d(t_1)$, $c(t_2) = c(s_2)$, and $R$ is a graphed equivalence relation.

A congruence $R$  on $\Se$ endows the set of equivalent classes with a semigroupoid structure, where 
$(\Se/R)^{(0)}$ is identified with $\Se^{(0)}$ and $\Se/R$ itself becomes the arrow set. Since the domain and codomain maps $d, c\colon \Se \to \Se^{(0)}$ are $R$-invariant, they factor through $\Se/R$ to maps $d, c\colon \Se/R \to \Se^{(0)}$ satisfying $d(\pi_R(s)) = d(s)$ and $c(\pi_R(s)) = c(s)$ for every $s \in \Se$, where $\pi_R$ is the canonical quotient map. This provides a graph structure to $\Se/R$. It is easily seen that $(\pi_R, \id_{\Se^{(0)}})$ defines a graph morphism between $\Se$ and $\Se/R$. 

Define the product on $\Se/R$ by $\pi_R(s)\cdot\pi_R(t):=\pi_R(st)$,
so that $(\pi_R,\mathrm{id}_{\Se^{(0)}})$ becomes a semigroupoid morphism. Because $R$ is a congruence, this multiplication is well‑defined and associative, and the resulting structure, called the \emph{quotient semigroupoid}, indeed forms a semigroupoid.

When $\Se$ is an inverse semigroupoid, the quotient semigroupoid $\Se/R$ is also an inverse semigroupoid, with involution given by $\pi_R(s)^* = \pi_R(s^*)$ for all $s \in \Se$, as stated in Proposition 4.5 of \cite{Luiz1}. Consequently, any congruence $R$ on an inverse semigroupoid respects the involution, meaning that $(s^*, t^*) \in R$ for all $(s, t) \in R$.  

Let $\Se$ be an inverse semigroupoid. We define a relation $\sigma$ on $\Se$ by $(s,t) \in \sigma$ if and only if there exists $r \in \Se$ such that $r \leqslant s, t$. In this fashion, $\sigma$ is a congruence on $\Se$ and $\Se / \sigma$ is a groupoid as follows.

\begin{proposition}
Let $\Se$ be an inverse semigroupoid and let $\sigma$ be the relation defined above. Then, $\sigma$ is a congruence on $\Se$, and $\Se/\sigma$ is a groupoid. Furthermore, if $\Ge$ is a groupoid and $\varphi\colon\Se\to\Ge$ is a morphism of semigroupoids, then there exists a unique morphism $\tilde{\varphi}\colon\Se/\sigma\to\Ge$ making the following diagram commute:
$$\xymatrix{\Se \ar[rr]^{\varphi} \ar[rd]_{\pi_{\sigma}}& & \Ge \\ & \Se/\sigma \ar@{..>}[ru]_{\tilde{\varphi}}}$$
\end{proposition}

\begin{proof}
It follows from Lemma 4.11, Proposition 4.14, and Proposition 4.16 of \cite{Luiz1}.
\end{proof}

The congruence $\sigma$ is called \emph{minimal groupoid congruence} on $\Se$ and the groupoid $\Se/\sigma$ is called \emph{maximal groupoid image} of $\Se$.

We prove an auxiliary result that establishes useful equivalences for determining when two elements are $\sigma$-congruent.

\begin{proposition}\label{sigma_equivalence}
Let $\Se$ be an inverse semigroupoid and let $s,t \in \Se$ be parallel arrows. The following conditions are equivalent:
\begin{enumerate}[label=(\roman*)]
    \item\label{ce1} $s \sigma t$.
    \item\label{ce2} There exists $e \in E(\Se)$ such that $(s,e), (t,e) \in \Se^{(2)}$ and $se = te$.
    \item\label{ce3} There exists $f \in E(\Se)$ such that $(f,s), (f,t) \in \Se^{(2)}$ and $fs = ft$.
\end{enumerate}
\end{proposition}

\begin{proof}
Assuming \ref{ce1}, let $r \leqslant s,t$ and let $e_1, e_2 \in E(\Se)$ be such that $(s, e_1), (t, e_2) \in \Se^{(2)}$ and $se_1 = r = te_2$. Since $c(e_1) = d(t) = d(s) = c(e_2)$, we have $(e_1, e_2) \in \Se^{(2)}$. Setting $e = e_1 e_2$, it is easily seen that $se = te$, thus concluding \ref{ce2}. The implication $\ref{ce1} \implies \ref{ce3}$ can be proved in a similar way, and the implications $\ref{ce2} \implies \ref{ce1}$ and $\ref{ce3} \implies \ref{ce1}$ are evident.
\end{proof}

Motivated by the equivalences just proved, we now single out those congruences on $\Se$ that cannot relate non‑idempotent arrows onto idempotents.

\begin{definition}\label{idempotent_pure_congruence}
Let $\Se$ be an inverse semigroupoid. A congruence $R$ on $\Se$ is said to be \emph{idempotent pure} if $(s,e)\in R$ and $e\in E(\Se)$ imply $s\in E(\Se)$.
\end{definition}

The following proposition shows that this idempotent‑purity condition admits several equivalent characterizations, linking the behavior of the quotient map to the involutive structure of $\Se$.

\begin{proposition}\cite[Proposition 4.20]{Luiz1}\label{idempotent_pure_equiv}
Let $\Se$ be an inverse semigroupoid and $R$ a congruence on $\Se$. The following conditions are equivalent:
\begin{enumerate}[label=(\roman*)]
    \item $R$ is idempotent pure;
    \item The canonical quotient map $\pi_{R}: \Se\to \Se/R$ is idempotent pure;
    \item If $(s,t)\in R$, then $s^*t\in E(\Se)$ (and also $st^*\in E(\Se)$).
\end{enumerate}
\end{proposition}

The next definition introduces $E$-unitary inverse semigroupoids, a generalization of the usual notion of $E$-unitary semigroup and the main object of study in this work. This notion has appeared in the literature in \cite{liu}, and in the context of inverse categories in \cite{steinberg_2003_lifting}.

\begin{definition}
An inverse semigroupoid $\Se$ is said to be \emph{$E$-unitary} if its minimal groupoid congruence is idempotent pure.
\end{definition}

Replacing $R$ with $\sigma$ in \cref{idempotent_pure_equiv}, we obtain the following result.

\begin{proposition}\label{e_unitary_equiv}
Let $\Se$ be an inverse semigroupoid. The following conditions are equivalent:  
\begin{enumerate}[label=(\roman*)]
\item $\Se$ is $E$-unitary, that is, $\sigma$ is idempotent pure;
\item The canonical quotient map $\pi_{\sigma}: \Se\to \Se/\sigma$ is idempotent pure;
\item If $(s,t)\in \sigma$, then $s^*t\in E(\Se)$ and $st^*\in E(\Se)$;
\item $(s,t)\in \sigma$ if and only if $s^*t\in E(\Se)$ and $st^*\in E(\Se)$;
\item If $s\in\Se$, $e\in E(\Se)$ and $e\leqslant s$, then $s\in E(\Se)$.
\end{enumerate}  
\end{proposition}

\begin{proof}
The equivalences (i)-(ii)-(iii) are direct from \cref{idempotent_pure_equiv}. To prove (iii) $\Rightarrow$ (iv), let $s,t\in\Se$ such that $s^*t\in E(\Se)$. Then  
$$
ss^*t=ss^*tt^*t=tt^*ss^*t=t(s^*t)^*(s^*t)=ts^*t
$$
and hence $(s,t)\in\sigma$.  

To prove (iv) $\Rightarrow$ (v), let $s\in\Se$ and $e\in E(\Se)$ such that $e\leqslant s$. Then $e\leqslant ss^*$ and hence $(ss^*,s)\in\sigma$. By item (iv), we have $s=ss^*s\in E(\Se)$.  

To prove (v) $\Rightarrow$ (i), let $s\in\Se$ and $e\in E(\Se)$ such that $(s,e)\in\sigma$. Then, there exists $r\in\Se$ such that $r\leqslant s,e$. Since $r\leqslant e$, we have $r\in E(\Se)$. By (v), we have $s\in E(\Se)$ (since $r\leqslant s$ and $r\in E(\Se)$).  
\end{proof}

Building on the above equivalences, we can now prove a technical lemma showing that any two $\sigma$‑congruent parallel arrows satisfy $st^*t = ts^*s$.

\begin{lemma}\label{lema-tecnico-idempotentes}
Let $\Se$ be an $E$-unitary inverse semigroupoid. If $s,t\in\Se$ are $\sigma$-congruent parallel arrows, then $st^*t = ts^*s$.
\end{lemma}

\begin{proof}
Initially, observe that since $(st^*t)^*(st^*t)=(ts^*s)^*(ts^*s)$, it follows that
$$st^*t=(st^*t)(st^*t)^*(st^*t)=(st^*t)(ts^*s)^*(ts^*s)=(st^*)(ts^*s)$$
and
$$ts^*s=(ts^*s)(ts^*s)^*(ts^*s)=(ts^*s)(st^*t)^*(st^*t)=(ts^*)(st^*t).$$
Because $\Se$ is $E$-unitary, we may deduce from \cref{e_unitary_equiv} that $st^*,ts^*\in E(\Se)$. Hence, $st^*t\leqslant ts^*s$ and $ts^*s\leqslant st^*t$, which further implies $st^*t = ts^*s$ as desired.
\end{proof}

We end this section with two examples of inverse semigroupoids. The first one is an adaptation from connected groupoids structures. Given a group $G$ and a non-empty set $X$, there is a well-known strategy to define a groupoid $X\times G\times X$ and every connected groupoid is isomorphic to one groupoid constructed in this way (see Proposition 3.3.6 of \cite{Lawson1998}). We adapt this construction replacing the group $G$ with an inverse semigroup $S$ to obtain an inverse semigroupoid as follows.

\begin{example}
Let $S$ be an inverse semigroup and $A$ a non-empty set. Then, the set $S_A = A \times S \times A$ admits a structure of an inverse semigroupoid, where the object set is given by $S_A^{(0)} = A$. The domain and codomain maps are defined by $d(v,s,u) = u$ and $c(v,s,u) = v$ for every $(v,s,u) \in S_A$.  

The composition is given by $ (w,t,v)(v,s,u) = (w,ts,u) $, for all $(w,t,v), (v,s,u) \in S_A$. The inverse of an element $(v,s,u) \in S_A$ is given by $ (v,s,u)^* = (u,s^*,v) $.  

Furthermore, it is straightforward to verify that an idempotent element has the form $(u,e,u)$ for some $u \in A$ and $e \in E(S)$. The natural partial order in $S_A$ is induced by the natural partial order on $S$, in the sense that $ (v,s,u) \leqslant (v,t,u) $ in $S_A$ if and only if $ s \leqslant t $ in $S$, for any pair of parallel arrows $(v,s,u), (v,t,u) \in S_A$.  

\end{example}

The next example is given by Cordeiro in \cite{Luiz1}. It is a special inverse subsemigroupoid of a particular case of the previous example. To establish notation, we denote by $\ix$ the inverse monoid of partial bijections on $X$, where the product is given by the partial composition of functions.

\begin{example}\label{ancora}
Let $\pi\colon X\to A$ be a surjection, and consider the following subset of $\ix_A$:  
$$\mathcal{J}(\pi)=\{(v,f,u)\mid \dom(f)\subseteq \pi^{-1}(u), \ran(f)\subseteq \pi^{-1}(v)\}.$$  
It is straightforward to verify that $\mathcal{J}(\pi)$ is closed under the product of $\ix_A$ and, therefore, forms a subsemigroupoid of $\ix_A$ with object space $A$. Furthermore, it is easy to see that $(u,f^{-1},v)\in\mathcal{J}(\pi)$ for every $(v,f,u)\in\mathcal{J}(\pi)$, which implies that $\mathcal{J}(\pi)$ is an inverse subsemigroupoid of $\ix_A$.  
\end{example}



\section{Partial Actions of Inverse Semigroupoids}\label{sec3}

This section studies partial actions of inverse semigroupoids. We begin with a set-theoretic framework for partial actions and recall the globalization construction from \cite{demeneghi-tasca-2024}. We then consider ordered partial actions on posets and prove that each such action admits a globalization.

We give the definition of partial morphism of inverse semigroupoids.  
\begin{definition}\label{partial_morphism}
Let $\Se$ and $\Te$ be inverse semigroupoids. A map $\varphi\colon\Se\to\Te$ is said to be a \emph{partial morphism} if the following conditions hold:  
\begin{enumerate}[label=(\roman*)]  
    \item\label{pmi} $\varphi(s^*)=\varphi(s)^*$;  
    \item\label{pmii} If $(s,t)\in \mathcal{S}^{(2)}$, then $(\varphi(s), \varphi(t))\in \Te^{(2)}$ and $\varphi(s)\varphi(t)\leqslant \varphi(st)$;  
    \item\label{pmiii} If $s\leqslant t$ in $\mathcal{S}$, then $\varphi(s)\leqslant \varphi(t)$ in $\Te$.  
\end{enumerate}  
\end{definition}

Any semigroupoid morphism $\varphi\colon\Se\to\Te$ naturally defines a partial morphism. By \cref{pmii} in \cref{partial_morphism}, there exists a unique map $\varphi^{(0)}\colon\Se^{(0)}\to\Te^{(0)}$ such that $(\varphi,\varphi^{(0)})$ forms a morphism of graphs. Thus, similarly to the case of semigroupoid morphisms, we may sometimes refer to $(\varphi,\varphi^{(0)})$ as a partial morphism to emphasize the object map $\varphi^{(0)}$.  


Given a partial morphism $\varphi\colon\Se\to\ix_A$, we denote by $\varphi_s\in\ix$ the second component of $\varphi(s)$ and by $X_s$ the range of $\varphi_s$ for each $s\in\Se$. Notice that, by \cref{pmi} of \cref{partial_morphism}, $\varphi_{s^*}$ is the inverse of $\varphi_s$, and consequently, the domain of $\varphi_s$ is given by $X_{s^*}$.  

\begin{definition}\label{partial_action_def1}
Let $\Se$ be an inverse semigroupoid and $X$ be a set. A \emph{partial action} of $\Se$ on $X$ is a partial morphism $\theta\colon\Se\to \ix_{\Se^{(0)}}$ such that the object map $\theta^{(0)}$ is the identity on $\Se^{(0)}$ and $X=\bigcup_{s\in\Se} X_s$. Furthermore, a partial action $\theta$ is said to be a \emph{global action} if it is a morphism of semigroupoids.  
\end{definition}

The requirement $\bigcup_{s\in\Se} X_s=X$ is not very restrictive since, in its absence, we can replace $X$ with the union of the ranges $X_s$, thereby obtaining an action of $\Se$ on $\bigcup_{s\in\Se} X_s$. Partial actions satisfying this additional requirement are sometimes called \emph{non-degenerate} partial actions. For convenience, we incorporate this condition directly into the definition.  

\begin{remark}\label{partial_action_def_3}
Since the object map $\theta^{(0)}$ is the identity on $\Se^{(0)}$, for any $s\in\Se$ we have $\theta(s)=(c(s),\theta_s,d(s))$, where $\theta_s\colon X_{s^*}\to X_s$. Hence, a partial action $\theta\colon\Se\to \ix_{\Se^{(0)}}$ induces a pair $\big(\{X_s\}_{s \in \Se}, \{\theta_s\}_{s \in \Se}\big)$ consisting of a collection $\{X_s\}_{s \in \Se}$ of subsets of $X$ and a collection $\{\theta_s\}_{s \in \Se}$ of maps $\theta_s \colon X_{s^*} \to X_s$ satisfying the following conditions:  
\begin{enumerate}[label=(E\arabic*)]  
    \item\label{e1} For every $s\in\mathcal{S}$, $\theta_s$ is a bijection and satisfies $\theta_s^{-1}=\theta_{s^*}$. Moreover, $X=\bigcup_{s\in\Se} X_s$;  
    \item\label{e2} $\theta_s\circ\theta_t \subseteq \theta_{st}$ for every $s, t \in \Se$ such that $(s,t)\in\mathcal{S}^{(2)}$;  
    \item\label{e3} $X_s \subseteq X_t$ for every $s,t \in \Se$ such that $s\leqslant t$.  
\end{enumerate}  
Furthermore, if $\theta$ is a global action, then  
\begin{enumerate}[label=(E\arabic*)]\setcounter{enumi}{3}  
    \item\label{eglob} $\theta_{st}=\theta_s\circ\theta_t$ for every $s, t \in \Se$ such that $(s,t)\in\mathcal{S}^{(2)}$ (i.e., equality holds in \ref{e2}).  
\end{enumerate}
On the other hand, given any pair $\big(\{X_s\}_{s \in \Se}, \{\theta_s\}_{s \in \Se}\big)$, such that $\{X_s\}_{s \in \Se}$ is a collection of subsets of $X$ and $\{\theta_s\}_{s \in \Se}$ is a collection of maps $\theta_s \colon X_{s^*} \to X_s$ satisfying conditions \ref{e1}, \ref{e2}, and \ref{e3},  induces a partial action $\theta$ of $\Se$ on $X$, given by $\theta(s)=(c(s),\theta_s,d(s))$. Moreover, this action is global if the initial pair also satisfies \ref{eglob}. Accordingly, we often regard the pair $\big(\{X_s\}_{s \in \Se}, \{\theta_s\}_{s \in \Se}\big)$ as the partial action $\theta\colon\Se\to \ix_{\Se^{(0)}}$ of $\Se$ on $X$.
\end{remark}

In Definition 2.4 of \cite{demeneghi-tasca-2024}, Demeneghi and Tasca introduced partial actions of inverse semigroupoids on sets as a pair $\big(\{X_s\}_{s \in \Se}, \{\theta_s\}_{s \in \Se}\big)$ satisfying a list of axioms different from that presented in \cref{partial_action_def_3}. However, they proved (Proposition 2.10 of \cite{demeneghi-tasca-2024}) that the two lists of axioms yield equivalent notions of partial action. For convenience, we explicitly state this equivalence in the next proposition.

\begin{proposition}\label{partial_action_def2}
Let $X$ be a set, $\Se$ be an inverse semigroupoid, $\{X_s\}_{s\in\Se}$ be a collection of subsets of $X$ and $\{\theta_s\}_{s\in\Se}$ a collection of maps $\theta_s\colon X_{s^*}\to X_s$. The pair $\theta=\big(\{X_s\}_{s\in\Se}, \{\theta_s\}_{s\in\Se}\big)$ is a partial action of $\Se$ on $X$ if and only if it satisfies  
\begin{enumerate}[label=(P\arabic*)]  
    \item\label{p1} $\theta_e=\id_{X_e}$ for all $e \in E(\Se)$. Moreover, for all $x \in X$, there exists $e \in E(\Se)$ such that $x \in X_e$;  
    \item\label{p2} $X_s \subseteq X_{ss^*}$ for every $s \in \Se$;  
    \item\label{p3} $\theta_t^{-1}(X_t \cap X_{s^*})=X_{(st)^*}\cap X_{t^*}$ for all $(s,t)\in\mathcal{S}^{(2)}$ and, moreover, $\theta_s\big(\theta_t(x)\big)=\theta_{st}(x)$ for all $x \in X_{(st)^*}\cap X_{t^*}$.  
\end{enumerate}  
Furthermore, if $\theta$ is a partial action of $\Se$ on $X$, then it is a global action if and only if 
\begin{enumerate}[label=(P\arabic*)]\setcounter{enumi}{3}  
    \item\label{pglob} $X_s = X_{ss^*}$ for every $s \in \Se$ (i.e., equality holds in \ref{p2}).  
\end{enumerate}  
\end{proposition}

In contexts where one must display the set $X$, we represent the partial action $\theta$ of $\Se$ on $X$ by the ordered pair $(\theta,X)$.

The following definition introduces the natural notion of equivalence for partial actions.

\begin{definition}\label{def:equivariant_map}
Let $\theta^X=(\{X_s\}_{s\in\Se}, \{\theta_s^X\}_{s\in\Se})$ and $\theta^Y=(\{Y_s\}_{s\in\Se}, \{\theta_s^Y\}_{s\in\Se})$ be partial actions of an inverse semigroupoid $\Se$ on sets $X$ and $Y$, respectively. A function $\varphi\colon X\rightarrow Y$ is said to be $\Se$\emph{-equivariant} if $\varphi(X_s)\subseteq Y_s$ for every $s\in\Se$ and, moreover, $\varphi\big(\theta_s^X(x)\big)=\theta_s^Y\big(\varphi(x)\big)$ for every $x \in X_{s^*}$. If moreover $\varphi$ is bijective and $\varphi^{-1}$ is also $\Se$-equivariant, we will say that $\varphi$ is an \emph{equivalence of partial actions}.
\end{definition}

Before turning to ordered partial actions of inverse semigroupoids on posets and their globalizations, we lay the groundwork with the set‑theoretic approach to globalization, starting with the definition of an orbit. 
If $\theta$ is a partial action of an inverse semigroupoid $\Se$ on a set $X$, and $Y\subseteq X$, we define the orbit of $Y$ under $\theta$ by
$$\orb(Y)=\bigcup_{s\in\Se} \theta_s(Y \cap X_{s^*}).$$

\begin{definition}\label{def:globalization}
Let $\theta$ be a partial action of an inverse semigroupoid $\Se$ on a set $X$. A pair $(\eta,\varphi)$ is said to be a \emph{globalization} of $\theta$ if $\eta$ is a global action of $\Se$ on a set $E$ and $\varphi\colon X \to E$ is an injective $\Se$-equivariant map such that
\begin{enumerate}[label=(\roman*)]
\item $\varphi$ induces an equivalence between $\theta$ and the restriction of $\eta$ to $\varphi(X)$, and
\item the orbit of $\varphi(X)$ under $\eta$ coincides with $E$.
\end{enumerate}
\end{definition}

To emphasize the underlying set $E$, we frequently denote a globalization of $(\theta,X)$ by the triple $(\eta,E,\varphi)$ rather than by $(\eta,\varphi)$.

Demeneghi and Tasca showed that every partial action of an inverse semigroupoid on a set admits a globalization. We summarize their set‑theoretic construction in the following remark; see \cite{demeneghi-tasca-2024} for full details.

\begin{remark}\label{remark:globalization}
Let $\theta=\big(\{X_s\}_{s\in\Se}, \{\theta_s\}_{s\in\Se}\big)$ be a partial action of an inverse semigroupoid $\Se$ on a set $X$. We briefly recall the construction of a globalization $(\eta,E,i)$ for $(\theta,X)$ due to Demeneghi and Tasca.

As a set, $E$ is the quotient of $D:=\{(s,x)\in \Se\times X \mid x \in X_{s^*s}\}$ by the equivalence relation $\approx$ generated by the reflexive and symmetric relation $\sim$ defined on $D$ which identifies pairs $(s,x)$ and $(t,y)$ in $D$ if either
\begin{enumerate}[label=(R\arabic*)]
    \item\label{r1} $(t^*,s)\in\Se^{(2)}$, $x \in X_{s^*t}$ and $\theta_{t^*s}(x)=y$; or
    \item\label{r2} $s, t \in E(\Se)$ and $x=y$.
\end{enumerate}

Sometimes we will write $(s,x)\simrum (t,y)$ and $(s,x)\simrdois (t,y)$ to specify by which item above we have $(s,x)\sim (t,y)$. We also denote by $[s,x]$ the equivalence class of an element $(s,x)\in D$.

For each $s\in\Se$, the set $E_s$ is defined as the image of the set $D_s=\{(p,x)\in D \mid (s^*,p)\in\Se^{(2)} \text{ and } x \in X_{p^*ss^*p}\}$ under the equivalence relation $\approx$ and the map $\eta_s\colon E_{s^*}\to E_s$ is defined by $\eta_s([p,x])=[sp,x]$ if $(p,x)\in D_{s^*}$. Finally, the $\Se$-equivariant map $i\colon X \to E$ is given by $i(x)=[e,x]$ in which $e\in E(\Se)$ is any idempotent such that $x\in X_e$.
\end{remark}

We now focus on ordered partial actions of inverse semigroupoids on partially ordered sets and on constructing their globalizations.

\begin{definition}\label{def:ordered_partial_action}
Let $\Se$ be an inverse semigroupoid and $X$ be a partially ordered set. A (set theoretical) partial action $\theta=\big(\{X_s\}_{s\in\Se}, \{\theta_s\}_{s\in\Se}\big)$ of $\Se$ on the underlying set $X$ is said to be an ordered partial action of $\Se$ on $X$ if each $X_s$ is an order ideal of $X$ and each $\theta_s$ is an order isomorphism.
\end{definition}

Henceforth, we abbreviate partially ordered set to poset. 
Notice that in terms of partial morphisms, an ordered partial action of an inverse semigroupoid $\Se$ on a poset $X$ is a partial action $\theta\colon\Se\to \ix^{\text{o}}_{\Se^{(0)}}$, in which $\ix^{\text{o}}$ denotes the inverse monoid consisting of order isomorphisms between order ideals of $X$. Accordingly, the expression ``partial action by order isomorphisms'' is often used to refer to ordered partial actions.

When no ambiguity arises, we drop “ordered” and simply speak of a partial action on a poset. In this section, however, we retain “ordered” to distinguish these from the underlying set-theoretic partial actions.

We introduce now the natural notion of order equivalence for ordered partial actions.

\begin{definition}\label{def:ordered_equivariant_map}
Let $\theta^X=(\{X_s\}_{s\in\Se}, \{\theta_s^X\}_{s\in\Se})$ and $\theta^Y=(\{Y_s\}_{s\in\Se}, \{\theta_s^Y\}_{s\in\Se})$ be ordered partial actions of an inverse semigroupoid $\Se$ on posets $X$ and $Y$, respectively. A function $\varphi\colon X\rightarrow Y$ is said to be \emph{ordered} $\Se$\emph{-equivariant} if it is $\Se$-equivariant in the sense of \cref{def:equivariant_map} and is order-preserving. If moreover $\varphi$ is bijective and $\varphi^{-1}$ is also ordered $\Se$-equivariant, we will say that $\varphi$ is an \emph{ordered equivalence of partial actions}.
\end{definition}

Let $X$ be a poset and let $\theta$ be an ordered partial action of the inverse semigroupoid $\Se$ on $X$. Denote by $(\eta,i)$ the set-theoretic globalization of $\theta$ constructed in Remark \ref{remark:globalization}. Our goal is to show that the globalization domain $E$ inherits a natural partial order from $X$ and that, with respect to this order, each $\eta_s$ is an order isomorphism. We begin with an auxiliary lemma that verifies how the original order on $X$ interacts with the equivalence relation defining $E$.

\begin{lemma}\label{lemma:tec}
Let $\Se$ be an inverse semigroupoid, $X$ be a partially ordered set and $\theta$ be a set-theoretical partial action of $\Se$ on the underlying set $X$. With the notation used in \cref{remark:globalization}, we have:
\begin{enumerate}[label=(\roman*),ref=(\roman*)]
    \item\label{lemma:tec:item1} Let $(s,x), (s,y) \in D$. If $(s,x)\approx(s,y)$, then $x=y$.
    \item\label{lemma:tec:item2} Let $(s,x), (t,y) \in D$ and $x'\in X$. If $(s,x)\approx (t,y)$ and $x'\leqslant x$, then there exists $y'\leqslant y$ such that $(s,x')\approx (t,y')$.
    \item\label{lemma:tec:item3} Let $(s,x), (t,y) \in D$. If there exists $(r,y')\in D$ and $x'\leqslant y'$ such that $(r,x')\approx (s,x)$ and $(r,y')\approx (t,y)$, then for all $(p,z)\in D$ such that $(p,z)\approx (t,y)$, there exists $z'\leqslant z$ such that $(p,z')\approx (s,x)$.
\end{enumerate}
\end{lemma}

\begin{proof}
To prove \cref{lemma:tec:item1}, let $(s,x), (s,y)\in D$ such that $(s,x)\approx (t,y)$. By Lemma 3.13 of \cite{demeneghi-tasca-2024}, we have $(s,x)\simrum (t,y)$. This amounts to say that $y=\theta_{s^*s}(x)=x$.

To prove \cref{lemma:tec:item2}, suppose initially that $(s,x)\sim (t,y)$. If $(s,x)\simrum (t,y)$, then $(t^*,s)\in\Se^{(2)}$, $x\in X_{s^*t}$ and $\theta_{t^*s}(x)=y$. Since $x' \leqslant x$, we have $x'\in X_{s^*t}$. Setting $y'=\theta_{t^*s}(x')$, we have $y'\leqslant y$ because $\theta_{t^*s}$ is order preserving and, moreover we have $(s,x')\simrum (t,y')$. Otherwise, if $(s,x)\simrdois (t,y)$, then $x=y$. So, we may just choose $y'=x'$ to obtain $(s,x')\simrdois (t,y')$. The general case $(s,x)\approx (t,y)$ now follows by a simple induction argument.

Finally, to prove \cref{lemma:tec:item3}, let $(p,z)\in D$ such that $(p,z)\approx (t,y)$. By hypothesis we have $(r,y')\approx (p,z)$ and $x'\leqslant y'$ and so we may deduce by \cref{lemma:tec:item2} that there exists $z'\leqslant z$ such that $(r,x') \approx (p,z')$. Thus, we have $(p,z') \approx (s,x)$ as desired.
\end{proof}

It follows that the order on $X$ induces a partial order on $E$, with respect to which the global action $\eta$ is order‐preserving.

\begin{proposition}\label{prop:induced_poset}
Let $\Se$ be an inverse semigroupoid, $X$ a partially ordered set, and $\theta$ a set-theoretic partial action of $\Se$ on the underlying set $X$. If $(\eta,E,i)$ is the globalization of $\theta$ described in \cref{remark:globalization}, then $E$ admits a partial order defined by declaring $[s,x]\leqslant [t,y]$ if and only if there exist $(r,y')\in D$ and $x'\in X$ such that $(r,y')\approx (t,y)$, $x'\leqslant y'$ and $(r,x')\approx (s,x)$. Moreover, with this order, the action $\eta$ of $\Se$ on $E$ is by order isomorphisms.
\end{proposition}

\begin{proof}
First, observe that \cref{lemma:tec:item3} of \cref{lemma:tec} can be interpreted as follows: if $[s,x]\leqslant [t,y]$, then for every representative $(p,z)$ of $[t,y]$, there exists $z'\leqslant z$ such that $(p,z')$ is a representative of $[s,x]$.

We now show that the given relation defines a partial order on $E$. Reflexivity is immediate from the definition of the relation. 

To prove antisymmetry, let $(s,x),(t,y) \in D$ such that $[s,x]\leqslant [t,y]$ and $[t,y]\leqslant [s,x]$. By \cref{lemma:tec:item3} of \cref{lemma:tec}, the first inequality ensures the existence of $y'\in X$ such that $y'\leqslant y$ and $(t,y')\approx (s,x)$. Similarly, applying \cref{lemma:tec:item3} again to the second inequality, there exists $y''\in X$ such that $y''\leqslant y'$ and $(t,y'')\approx (t,y)$. By \cref{lemma:tec:item1} of \cref{lemma:tec}, we have $y=y''$. Since $y''\leqslant y' \leqslant y$, it follows that $y=y'$, and thus $(t,y)\approx (s,x)$.

For the first part, it remains to prove transitivity. Let $(s,x), (t,y), (r,z)\in D$ such that $[s,x]\leqslant [t,y]\leqslant [r,z]$. By \cref{lemma:tec:item3} of \cref{lemma:tec}, there exists $z'\in X$ such that $z'\leqslant z$ and $(r,z')\approx (t,y)$. Since $[s,x]\leqslant [r,z']$, we may apply \cref{lemma:tec:item3} once more to obtain $z''\in X$ such that $z''\leqslant z'$ and $(r,z'')\approx (s,x)$. This shows that $[s,x]\leqslant [r,z]$, as desired.

It remains to show that $\eta$ acts on $E$ by order isomorphisms. We begin by proving that, for every $s \in \Se$, the set $E_s$ is an order ideal of $E$. Let $[p,x], [q,y] \in E$ with $[p,x] \in E_s$ and $[q,y] \leqslant [p,x]$. By the definition of $E_s$ (see \cref{remark:globalization}), the assumption $[p,x] \in E_s$ implies the existence of $(r,z) \in D$ satisfying $(p,x) \approx (r,z)$, $(s^*,r) \in \Se^{(2)}$, and $z \in X_{r^*ss^*r}$. Since $[q,y] \leqslant [r,z]$, it follows from \cref{lemma:tec:item3} of \cref{lemma:tec} that there exists $z' \leqslant z$ such that $(q,y) \approx (r,z')$. In summary, we have $(r,z')\in D$ such that $(q,y)\approx (r,z')$, $(s^*,r)\in\Se^{(2)}$ and $z'\in X_{r^*ss^*r}$, which means that $[q,y]\in E_s$ as desired.

Finally, we show that $\eta_s$ is order-preserving for every $s \in \Se$. Let $[p,x], [q,y] \in E_{s^*}$ with $[p,x] \leqslant [q,y]$. Since $[q,y] \in E_{s^*}$, there exists $(r,z) \in D$ such that $(r,z) \approx (q,y)$, $(s,r) \in \Se^{(2)}$, and $z \in X_{r^*s^*sr}$. As $[p,x] \leqslant [r,z]$, \cref{lemma:tec:item3} of \cref{lemma:tec} yields $z' \leqslant z$ such that $(p,x) \approx (r,z')$. Since $z' \in X_{r^*s^*sr}$, we may use $(r,z')$ as a representative of $[p,x]$ in $D_s$, and therefore,
$$\eta_s([p,x]) = \eta_s([r,z']) = [sr,z'] \leqslant [sr,z] = \eta_s([r,z]) = \eta_s([q,y]),$$
which shows that $\eta_s$ is order-preserving, completing the proof.
\end{proof}

Having established the key lemmas, we are ready to introduce the notion of an ordered globalization and to prove that each ordered partial action admits one.

\begin{definition}\label{def:ordered_globalization}
Let $\theta$ be an ordered partial action of an inverse semigroupoid $\Se$ on a partially ordered set $X$. A pair $(\eta,\varphi)$ is said to be an \emph{ordered globalization} of $\theta$ if $\eta$ is an ordered global action of $\Se$ on a partially ordered set $E$ and $\varphi\colon X \to E$ is an injective ordered $\Se$-equivariant map such that
\begin{enumerate}[label=(\roman*)]
\item $\varphi$ induces an ordered equivalence between $\theta$ and the restriction of $\eta$ to $\varphi(X)$, and
\item $\varphi(X)$ is an order ideal of $E$ and the orbit of $\varphi(X)$ under $\eta$ coincides with $E$.
\end{enumerate}
\end{definition}

We are now in a position to prove that every ordered partial action admits such a globalization.

\begin{theorem}\label{o_teorema_1}
Let $\Se$ be an inverse semigroupoid. Then, every ordered partial action $\theta$ of $\Se$ on a partially ordered set $X$ admits an ordered globalization.
\end{theorem}

\begin{proof}
Viewing $\theta$ as a set-theoretic partial action of $\Se$ on the underlying set $X$, let $(\eta,E,i)$ denote the globalization of $\theta$ described in \cref{remark:globalization}. By \cref{prop:induced_poset}, $E$ admits a partial order and $\eta$ is an ordered global action of $\Se$ on $E$.

We now prove that $i(X)$ is an order ideal of $E$. Let $x \in X$ and $[s,y] \in E$ be such that $[s,y]\leqslant i(x)$. Choose $e \in E(\Se)$ such that $x \in X_e$, so that $i(x)=[e,x]$ and hence $[s,y] \leqslant [e,x]$. By \cref{lemma:tec:item3} of \cref{lemma:tec}, there exists $x'\leqslant x$ such that $(s,y)\approx (e,x')$. It follows that $[s,y]=[e,x']=i(x')\in i(X)$, as desired.

It remains to show that $X$ is order isomorphic to $i(X)$. It follows immediately from the definitions of the partial order on $E$ in \cref{prop:induced_poset} and of $i$ in \cref{remark:globalization} that $i$ is an order-preserving map. Now, let $x,y \in X$ satisfy $i(x)\leqslant i(y)$. By choosing $e,f\in E(\Se)$ such that $x\in X_e$ and $y \in X_f$, we have $[e,x]=i(x)\leqslant i(y)=[f,y]$. By \cref{lemma:tec:item3} of \cref{lemma:tec}, there exists $y'\leqslant y$ such that $(f,y')\approx (e,x)$. By Lemma 3.12 of \cite{demeneghi-tasca-2024}, it follows that $y'=\theta_f(y')=\theta_e(x)=x$, thus $x=y'\leqslant y$, as desired.
\end{proof}

We can further refine \cref{o_teorema_1} by showing that the constructed globalization of $\theta$ has a universal property.

\begin{proposition}\label{prop:universal_globalization}
Let $\theta$ be an ordered partial action of an inverse semigroupoid on a partially ordered set $X$. The ordered globalization $(\eta,E,i)$ of $\theta$ given in the proof of \cref{o_teorema_1} is universal in the following sense: for any ordered global action $(\zeta,F)$ and any ordered $\Se$-equivariant map $j\colon X \to F$, there exists a unique ordered $\Se$-equivariant map $k\colon E\to F$ such that the diagram
$$\xymatrix{X \ar[rr]^{i} \ar[rd]_{j}& & E \ar@{..>}[ld]^{k}\\ & F &}$$
commutes.
\end{proposition}

\begin{proof}
At the set-theoretic level, Theorem~3.15 of \cite{demeneghi-tasca-2024} ensures the existence of a mediating function $k\colon E \to F$, necessarily given by
$k([s,x])=\zeta_s(j(x))$. It remains to verify that $k$ is order-preserving.

Let $[s,x], [t,y]\in E$ with $[s,x]\leqslant [t,y]$. By \cref{lemma:tec:item3} of \cref{lemma:tec}, there exists $y'\leqslant y$ such that $(t,y')\approx (s,x)$. Since $j$ is order-preserving, it follows that $j(y')\leqslant j(y)$. As $\zeta_t$ is also order-preserving, we conclude that
$$k([s,x])=k([t,y'])=\zeta_t(j(y'))\leqslant\zeta_t(j(y))=k([t,y]).$$
Therefore, $k$ is order-preserving, as desired.
\end{proof}



\section{P-theorem for inverse semigroupoids}\label{P-theorem}

We now turn to the main structural result of this paper. First, we recall the definition of a $P$-semigroup in the single-object case; next, we adapt this construction to the multi-object setting of inverse semigroupoids; and finally we prove that an inverse semigroupoid is $E$-unitary precisely when it arises from this construction.

A McAlister triple $(G,E,X)$ consists of a group $G$, a partially ordered set $E$, an order ideal $X$ of $E$ that is a meet semilattice under the induced order, together with a global action $\eta$ of $G$ on $E$ by order automorphisms, satisfying $\orb(X)=E$ and $\eta_g(X)\cap X \neq\emptyset$ for all $g \in G$.

The key ingredient in constructing a $P$-semigroup is the McAlister triple, which we now recall. From every McAlister triple we may construct an $E$-unitary inverse semigroup $P(G,E,X)$ as follows: as a set, define $P(G,E,X)=\{(g,x)\in G\times X \mid x \in X_{g^{-1}}\}$. The product is given by $(g,x)\cdot (h,y)=(gh,\eta_{h^{-1}}(x)\cdot y)$, in which $a\cdot b$ denotes the meet of $a$ and $b$ in $X$. An $E$-unitary inverse semigroup obtained in this way from a McAlister triple is said to be a $P$-semigroup. The McAlister $P$-theorem then asserts that every $E$-unitary inverse semigroup is isomorphic to a $P$-semigroup.

To extend this construction to the multi-object setting, we replace the group $G$ by a groupoid $\Ge$ acting by order isomorphisms on a poset $E$, and we upgrade the semilattice $X$ to a compatible family of semilattices indexed by the objects of $\Ge$.

\begin{definition}\label{def_semilatticeoid}
An inverse semigroupoid $X$ is said to be a \emph{(meet) semilatticeoid} if $X = E(X)$, meaning that every element of $X$ is idempotent.
\end{definition}

A semilatticeoid naturally decomposes into a disjoint union of meet semilattices. More precisely, if $X$ is a semilatticeoid, then for each $u \in X^{(0)}$, the set $X_u := d^{-1}=c^{-1}(u)$ forms a (meet) semilattice with the meet operation given by the product in $X$. Thus, $X$ can be expressed as the disjoint union $\bigsqcup_{u\in X^{(0)}} X_u$. Conversely, if $X=\bigsqcup_{u\in U} X_u$ is a disjoint union of meet semilattices, then is easy to verify that $(X,U)$ is a semilatticeoid with product defined by the meet operation.

\begin{remark}
Let $\theta=(\{X_s\}_{s\in\Se},\{\theta_s\}_{s\in\Se})$ be an ordered partial action of an inverse semigroupoid $\Se$ on a semilatticeoid $X$ (equipped with the natural partial order). Then, for each $s\in\Se$, $X_s$ is an ideal of $X$ in the sense that if $x\in X$, $y\in X_s$ and $(x,y)\in X^{(2)}$, then $xy \in X_s$. Indeed, since $xy\leqslant x$ and $X_s$ is an order ideal, it follows that $xy \in X_s$. 

Moreover, each map $\theta_s$ is a morphism of semigroupoids. That is, for all $x,y\in X_{s^*}$ such that $(x,y)\in X^{(2)}$, we have $(\theta_s(x),\theta_s(y))\in X^{(2)}$ and $\theta_s(xy)=\theta_s(x)\theta_s(y)$. To see this, note that $xy\leqslant x,y$ implies $\theta_s(xy)\leqslant \theta_s(x),\theta_s(y)$, since $\theta_s$ is order preserving. Hence, $(\theta_s(x),\theta_s(y))\in X^{(2)}$ and $$\theta_s(xy)\leqslant\theta_s(x)\theta_s(y).$$ 
Replacing $s$ by $s^*$, $x$ by $\theta_s(x)$ and $y$ by $\theta_s(y)$ in the last inequality we obtain $\theta_{s^*}\big(\theta_s(x)\theta_s(y)\big)\leqslant xy$. On the other hand, applying the order-preserving map $\theta_{s^*}$ to both sides of the highlighted inequality we obtain $xy\leqslant\theta_{s^*}(\theta_s(x)\theta_s(y))$. Thus, equality holds throughout, and we conclude that $\theta_s(xy) = \theta_s(x)\theta_s(y)$, as claimed.
\end{remark}

Combining the groupoid action on $E$ with the semilatticeoid structure, we can now define a McAlister triple in the multiobject setting.

\begin{definition}\label{def:Mclister_triple}
A \emph{McAlister triple} $(\Ge,E,X)$ consists of a groupoid $\Ge$, a partially ordered set $E$, an order ideal $X$ of $E$ which is a semilatticeoid under the induced order, together with an ordered global action $\eta$ of $\Ge$ on $E$ satisfying $\orb(X)=E$ and $\eta_g(X \cap E_{g^{-1}})\cap X\neq \emptyset$ for all $g \in \Ge$.
\end{definition}

When $(\Ge,E,X)$ is a McAlister triple, then  $\eta$ induces an ordered partial action $\theta$ of $\Ge$ on $X$ such that each $X_g$ is non-empty. Moreover, it is clear that $(\eta,E,i)$ is an ordered globalization of $(\theta,X)$, in which $i\colon X \to E$ is the inclusion map. Conversely, if $\theta$ is a partial action of $\Ge$ on a semilatticeoid $X$ such that each $X_g$ is non-empty, we can build a McAlister triple as follows.

\begin{theorem}\label{TMP}
Let $\theta$ be a ordered partial action of a groupoid $\Ge$ on a semilatticeoid $X$ such that each $X_g$ is non-empty. If $(\eta,E,i)$ is the ordered globalization in the proof of \cref{o_teorema_1}, then $(\Ge,E,i(X))$ is a McAlister triple.
\end{theorem}

\begin{proof}
Being $(\eta,E,i)$ an ordered globalization of $(\theta,X)$, we know that $E$ is a poset, $\eta$ is an ordered global action of $\Ge$ on $E$, $i(X)$ is an ordered ideal of $E$, the orbit of $i(X)$ under $\eta$ equals $E$ and $i(X)$ is a semilatticeoid in the induced order since $X$ and $i(X)$ are ordered isomorphic. It only remains to prove that $g\cdot i(X)\cap i(X)\neq\emptyset$.

Let $x \in X_{g^{-1}}$. Since $i$ a $\Ge$-equivariant map, we have $i(x)\in E_{g^{-1}}$. As $X_{g^{-1}}\subseteq X_{g^{-1}g}$, we may write $i(x)=[g^{-1}g,x]$. Then $\eta_s(i(x))=\eta_s([g^{-1}g,x])=[g,x]$. On the other hand, since $\theta_g(x)\in X_g\subseteq X_{gg^{-1}}$, we have $i(\theta_g(x))=[gg^{-1},\theta_g(x)]=[g,x]$. Therefore, $[g,x]\in g\cdot i(X)\cap i(X)$.
\end{proof}

As in Kellendonk and Lawson \cite{KL2004}, this theorem demonstrates that McAlister triples correspond exactly to partial actions of groupoids on semilatticeoids.

We now turn to the construction of the $P$-semigroup associated to a McAlister triple. Guided by Theorem \ref{TMP}, we reformulate this in terms of partial actions of groupoids on semilatticeoids rather than in terms of triples. Accordingly, from now on whenever \(\theta\) is an ordered partial action of an inverse semigroupoid \(\Se\) on a poset \(X\), we assume \(X_s\neq\varnothing\) for every \(s\in\Se\).

\begin{theorem}\label{theo:semidirect_product}
Let $\theta$ be an ordered partial action of an inverse semigroupoid $\Se$ on a semilatticeoid $X$. Then 
\begin{equation}\label{semidirect_product_set}
\Se\ltimes_{\theta} X=\left\{(s,x)\in\Se\times X\mid x\in X_{s^*}\right\},
\end{equation}
admits a structure of inverse semigroupoid, henceforward called the semidirect product of $\Se$ and $X$, with partially defined product
\begin{equation}\label{semidirect_product_product}
(s,x)(t,y)=\Big(st,\theta_{t^*}\big(x\theta_t(y)\big)\Big)
\end{equation}
defined whenever $(s,t)\in\Se^{(2)}$ and $(x,\theta_t(y))\in X^{(2)}$, and involution
\begin{equation}\label{semidirect_product_involution}
(s,x)^*=(s^*,\theta_s(x))
\end{equation}
for all $(s,x)\in\Se\ltimes_{\theta} X$. Moreover, if $\Se$ is $E$-unitary, then so is $\Se\ltimes_{\theta} X$.
\end{theorem}

\begin{proof}
First of all, notice that the right-hand side of \eqref{semidirect_product_product} is a well-defined element of $\Se\ltimes X$. Indeed, since $\theta_t(y)\in X_t$ and $X_t$ is an (order) ideal of $X$, it follows that $x\theta_t(y) \in X_t$. Similarly, since $x\in X_{s^*}$, it follows that $x\theta_t(y) \in X_{s^*}$. By \ref{p3}, we have $\theta_{t^*}(x\theta_t(y))\in \theta_{t^*}(X_t\cap X_{s^*})\subseteq X_{(st)^*}$.

We show now that $\Se\ltimes_{\theta} X$ admits a graph structure compatible with the partially defined product in \eqref{semidirect_product_product}, with object set $(\Se\ltimes_{\theta} X)^{(0)}=\Se^{(0)}\times X^{(0)}$ and domain and codomain maps $d,c\colon\Se\ltimes_{\theta} X \to\Se^{(0)}\times X^{(0)}$ given by $d(s,x)=\big(d(s),d(x)\big)$ and $c(s,x)=\big(c(s),c(\theta_s(x))\big)$ for every $(s,x) \in \Se\ltimes_{\theta} X$. Indeed, let $\big((s,x),(t,y)\big)\in (\Se\ltimes_{\theta} X)^{(0)}$. By the previous paragraph, we have that $\theta_{t^*}(x\theta_t(y))$ lies in $\theta_{t^*}(X_t\cap X_{s^*})$, which is the domain of $\theta_s\circ\theta_t$. Hence, $\theta_{st}\big(\theta_{t^*}(x\theta_t(y))\big)=\theta_s(x\theta_t(y))$. Since $x$ is idempotent, and both $x$ and $x\theta_t(y)$ lie in $X_{s^*}$, and since $\theta_s$ is order-preserving, we have $\theta_s(x\theta_t(y))=\theta_s(x)\theta_s(x\theta_t(y))$. Similarly, using the order-preserving property of $\theta_{t^*}$ and the idempotency of $y$, we obtain $\theta_{t^*}(x\theta_t(y))=\theta_{t^*}(x\theta_t(y))\theta_{t^*}(\theta_t(y))=\theta_{t^*}(x\theta_t(y))y$. Therefore,
$$d\big((s,x)\cdot (t,y)\big)
=d\big(st,\theta_{t^*}(x\theta_t(y))\big)
=\Big(d(st),d\big(\theta_{t^*}(x\theta_t(y))y\big)\Big)
=(d(t),d(y))
=d(t,y)$$
and
$$c\big((s,x)\cdot (t,y)\big)
=c\big(st,\theta_{t^*}(x\theta_t(y))\big)
=\Big(c(st),c\big(\theta_s(x)\theta_s(x\theta_t(y))\big)\Big)
=\big(c(s),c(\theta_s(x))\big)
=c(s,x),$$
as required.

In order to conclude that $\Se\ltimes_{\theta} X$ is a semigroupoid, it remains only to verify the associativity of the product. To this end, let $(s,x),(t,y),(r,z) \in\Se\ltimes_{\theta} X$ such that $\big((s,x),(t,y)\big)\in (\Se\ltimes_{\theta} X)^{(0)}$ and $\big((t,y),(r,z)\big)\in (\Se\ltimes_{\theta} X)^{(0)}$. Then $(s,t), (t,r)\in\Se^{(2)}$ and $(x,\theta_t(y)),(y,\theta_r(z))\in X^{(2)}$. Since $\Se$ is a semigroupoid, we have $(st)r=s(tr)$. Therefore, it remains to verify that
$$\theta_{r^*}\Big(\theta_{t^*}\big(x\theta_t(y)\big)\theta_r(z)\Big)
=\theta_{(tr)^*}\Big(x\theta_{tr}\big(\theta_{r^*}(y\theta_r(z))\big)\Big).$$
First, notice that both sides of this equality are well defined by the compatibility of the graph structure of $\Se\ltimes_{\theta} X$ with the product.

Since both $x\theta_t(y)$ and $\theta_t(y)$ lie in $X_t$, $\theta_{t^*}$ is order-preserving and $\theta_t(y)$ is idempotent, it follows that
$$
\theta_{t^*}\big(x\theta_t(y)\big)
= \theta_{t^*}\big(x\theta_t(y)\big)\theta_{t^*}(\theta_t(y))
= \theta_{t^*}\big(x\theta_t(y)\big)y.
$$
Furthermore, since $y\theta_r(z)$ lies in $X_{t^*}$, $y$ is idempotent and both $\theta_t$ and $\theta_{t^*}$ are order-preserving, we may write
$$
\theta_{t^*}\big(x\theta_t(y)\big)y\theta_r(z)
=\theta_{t^*}\big(x\theta_t(y)\big)\theta_{t^*}\big(\theta_t(y\theta_r(z))\big)
=\theta_{t^*}\big(x\theta_t(y)\theta_t(y\theta_r(z))\big)
=\theta_{t^*}\big(x\theta_t(y\theta_r(z))\big).
$$
Thus $x\theta_t(y\theta_r(z))$ lies in the domain of $\theta_{r^*}\circ\theta_{t^*}$ and we obtain
$$
\theta_{r^*}\Big(\theta_{t^*}\big(x\theta_t(y)\big)\theta_r(z)\Big)
=\theta_{r^*}\Big(\theta_{t^*}\big(x\theta_t(y)\big)y\theta_r(z)\Big)
=\theta_{r^*}\Big(\theta_{t^*}\big(x\theta_t(y\theta_r(z))\big)\Big)
=\theta_{(tr)^*}\big(x\theta_t(y\theta_r(z))\big).
$$
Finally, notice that $y\theta_r(z)=\theta_r(\theta_{r^*}(y\theta_r(z)))$ and $\theta_{r^*}(y\theta_r(z))$ lies in the domain of $\theta_t\circ\theta_r$, so that $x\theta_t(y\theta_r(z))=x\theta_t\big(\theta_r(\theta_{r^*}(y\theta_r(z)))\big)=x\theta_{tr}\big(\theta_{r^*}(y\theta_r(z))\big)$ and moreover
$$
\theta_{r^*}\Big(\theta_{t^*}\big(x\theta_t(y)\big)\theta_r(z)\Big)
=\theta_{(tr)^*}\big(x\theta_t(y\theta_r(z))\big)
=\theta_{(tr)^*}\Big(x\theta_{tr}\big(\theta_{r^*}(y\theta_r(z))\big)\Big),
$$
as desired.

We now prove that $\Se\ltimes_{\theta} X$ is an inverse semigroupoid. First, observe for every $(s,x)\in\Se\ltimes_{\theta} X$, we have $(s,x) (s^*,\theta_s(x)) (s,x)=(s,x)$, so each element admits a pseudoinverse. According to Theorem 2.1.1 of \cite{liu}, it suffices to verify that $E(\Se\ltimes_{\theta} X)$ is commutative to finish the argument. Note that the idempotent elements of $\Se\ltimes_{\theta} X$ are precisely those of the form $(e,x)\in\Se\ltimes_{\theta} X$ with $e\in E(\Se)$. Let $(e,x),(f,y)\in\Se\ltimes_{\theta} X$ be such that $e,f\in E(\Se)$ and $\big((e,x),(f,y)\big)\in (\Se\ltimes_{\theta} X)^{(2)}$. Then $(e,f)\in\Se^{(2)}$, $(x,y)\in X^{(2)}$ and $(e,x)(f,y)=(ef,xy)=(fe,yx)=(f,y)(e,x)$.

Finally, to prove that $\Se\ltimes_{\theta} X$ is $E$-unitary, let $(s,x),(e,y)\in\Se\ltimes_{\theta} X$ be parallel arrows such that $(e,y)\in E(\Se)$ and $(e,y)\leqslant (s,x)$. Then, we have $(s,e)\in \Se^{(2)}$, $(x,y)\in X^{(2)}$, and $(e,y)=(s,x)(e,y)=(se,xy)$. It follows that $s=se$, hence $e\leqslant s$. Since $\Se$ is $E$-unitary, it follows from \cref{e_unitary_equiv} that $s\in E(\Se)$. Therefore, $(s,x)\in E(\Se\ltimes_{\theta} X)$, and once again by \cref{e_unitary_equiv}, we conclude that $\Se\ltimes_{\theta} X$ is $E$-unitary.
\end{proof}

By specializing the previous theorem to the case where $\Se$ is a groupoid $\Ge$, we obtain the following corollary immediately.

\begin{corollary}\label{prop:P-semigroup_is_E-unitary}
Let $\theta$ be an ordered partial action of groupoid $\Ge$ on a semilatticeoid $X$. Then, the semidirect product $\Ge\ltimes X$ is an $E$-unitary inverse semigroupoid.
\end{corollary}
 
We say that an $E$-unitary inverse semigroupoid is a $P$-semigroupoid if it is isomorphic to a semidirect product obtained from an ordered partial action of a groupoid on a semilatticeoid.

To verify that the product in Definition \ref{semidirect_product_product} of Theorem \ref{theo:semidirect_product} behaves correctly, note that in the classical McAlister triple $(G,E,X)$ with a global action $\eta$, note that in the classical McAlister triple $\eta_{h^{-1}}(x)y$ as
$$\eta_{h^{-1}}(x)y=\eta_{h^{-1}}(x)\eta_{h^{-1}}(\eta_h(y))=\eta_{h^{-1}}(x\eta_h(y)),$$
making explicit the compatibility in both products.

In order to establish that $E$-unitary inverse semigroupoids coincide with $P$-semigroupoids, we introduce the Munn action—an intrinsic action of the semigroupoid on the semilattice of its idempotents.

\begin{example}\label{Munn-global}
Let $\Se$ be an inverse semigroupoid. For each $s\in\Se$ we set $X_s := \{e\in E(\Se) \mid e\leqslant ss^*\}$. Clearly, $X_s$ is an order ideal of $E(\Se)$. Moreover, if $e\in X_{s^*}$, then $(e,s^*), (s,e)\in \Se^{(2)}$ and $ses^*\leqslant s(s^*s)s^*=ss^*$, so $ses^*\in X_s$. Therefore, for each $s\in\Se$ we have a well-defined function $\theta_s\colon X_{s^*} \to X_{s}$ given by $\theta_s(e)=ses^*$, which is clearly order-preserving.

The pair $\theta=(\{X_s\}_{s\in \Se}, \{\theta_s\}_{s\in \Se})$ defines an ordered global action of $\Se$ on $E(\Se)$. To verify this, we check \cref{p1,p2,p3,pglob} from \cref{partial_action_def2}.
	
For \ref{p1}, let $e \in E(\Se)$. If $f\in X_e$, then $(e,f), (f,e)\in \Se^{(2)}$, $ef=f=fe$ and $\theta_e(f)=efe=f$. This proves that $\theta_e$ is the identity map on $X_e$. Also, it is easily seen that $e \in X_e$ since $e \leqslant e$.

Next, note that $X_s=X_{ss^*}$ for all $s\in\Se$, so \cref{p2,pglob} are immediately satisfied. 

It only remains to verify \ref{p3}. For this purpose, let $s,t\in\Se$ with $(s,t)\in\Se^{(2)}$. If $e \in X_{(st)^*} \cap X_{t^*}$, then $e\leqslant t^*s^*st$ and $\theta_t(e)=tet^*\leqslant t(t^*s^*st)t^*=tt^*s^*s\leqslant tt^*,s^*s$, which ultimately means that $e \in \theta_t^{-1}(X_t\cap X_{s^*})$. Conversely, if $e \in \theta^{-1}_t(X_t\cap X_s)$, then $tet^*=\theta_t(e) \leqslant tt^*,s^*s$ and, since $e\in X_{t^*}$, we have $e=t^*tet^*t\leqslant t^*t,t^*s^*st$ which amounts to say that $e \in X_{(st)^*} \cap X_{t^*}$.
\end{example}

Let $\Se$ be an $E$-unitary inverse semigroupoid.  The following proposition will yield an ordered partial action of the quotient $\Se/\sigma$ on the idempotent semilattice $E(\Se)$, arising naturally from the Munn action of $\Se$.
. 

\begin{proposition}\label{partMunn}
Let $\Se$ be an $E$-unitary inverse semigroupoid and $X$ be a partially ordered set. Then, every ordered global action $\theta=(\{X_s\}_{s\in \Se}, \{\theta_s\}_{s\in \Se})$ of $\Se$ on $X$ induces an ordered partial action $\alpha=(\{D_{\pi_{\sigma}(s)}\}_{s\in \Se}, \{\alpha_{\pi_{\sigma}(s)}\}_{s\in \Se})$ of $\Se/\sigma$ on $X$ such that $\alpha_{\pi_{\sigma}(s)}(x)=\theta_s(x)$ for every $x \in X_{s^*}$.
\end{proposition}

\begin{proof}
First of all, notice that if $s,t\in\Se$ are $\sigma$-congruent and $x\in X_{s^*}\cap X_{t^*}$, then setting $f=t^*ts^*s$ we have $x\in X_f=X_{s^*s}\cap X_{t^*t}$ and, by \cref{lema-tecnico-idempotentes}, that
$$sf = st^*ts^*s = ss^*st^*t = st^*t = tss^* = tt^*tss^* = tf,$$
which further gives
$$\theta_s(x)=\theta_s(\theta_f(x))=\theta_{se}(x)=\theta_{te}(x)=\theta_t(\theta_f(x))=\theta_t(x).$$
This implies that, for each $\sigma$-class, the maps $\theta_s$ with $s$ ranging over that class, can be glued to form a well-defined map. More specifically, for each $s\in\Se$, setting $D_{\pi_{\sigma}(s)}=\bigcup_{(s,t)\in\sigma} X_t$ we have a well-defined map $\alpha_{\pi_{\sigma}(s)}\colon D_{\pi_{\sigma}(s^*)}\to D_{\pi_{\sigma}(s)}$ given by $\alpha_{\pi_{\sigma}(s)}(x)=\theta_t(x)$ in which $t\in\Se$ is any element $\sigma$-congruent to $s$ with $x\in X_{t^*}$.

It is clear that each $D_{\pi_{\sigma}(s)}$ is an order ideal of $X$. Now let $x,y \in D_{\pi_{\sigma}(s)}$ with $x\leqslant y$. Since $y\in D_{\pi_{\sigma}(s)}$, there exists $t\in\Se$ such that $(s,t)\in\sigma$ and $y\in X_{t^*}$. As $X_{t^*}$ is an order ideal of $X$, we have $x\in X_{t^*}$. Therefore, since $\theta_t$ order-preserving, we obtain $\alpha_{\pi_{\sigma}(s)}(x)=\theta_t(x)\leqslant \theta_t(y)=\alpha_{\pi_{\sigma}(s)}(y)$, proving that $\alpha_{\pi_{\sigma}(s)}$ is order-preserving.

We now prove that $\alpha=(\{D_{\pi_{\sigma}(s)}\}_{s\in \Se},\{\alpha_{\pi_{\sigma}(s)}\}_{s\in \Se})$ is an ordered partial action of $\Se/\sigma$ on $X$. Indeed, we are going to verify \cref{p1,p2,p3,pglob} from \cref{partial_action_def2}.

For \ref{p1}, let $e \in E(\Se)$. If $x\in D_{\pi_{\sigma}(e)}$, then there exist $s\in\Se$ such that $(s,e)\in\sigma$ and $x\in X_{s^*}$. Since $\Se$ is $E$-unitary, we must have $s\in E(\Se)$ and so $\alpha_{\pi_{\sigma}(e)}(x)=\theta_s(x)=x$. Moreover, if $x\in X$, then there exists $e\in E(\Se)$ such that $x\in X_e\subseteq D_{\pi_{\sigma}(e)}$.

For \ref{p2}, let $s\in\Se$. If $t\in\Se$ is such that $(s,t)\in\sigma$, we must have $(ss^*,tt^*)\in\sigma$ because $\sigma$ is a congruence. This implies that $X_t=X_{tt^*}\subseteq D_{\pi_{\sigma}(ss^*)}$ for any $t\in\Se$ such that $(s,t)\in\sigma$ and, so $D_{\pi_{\sigma}(s)}=\bigcup_{(s,t)\in\sigma} X_t\subseteq D_{\pi_{\sigma}(ss^*)}$.

Finally, let $(s,t)\in\Se^{(2)}$ and lets argue that $\alpha^{-1}_{\pi_{\sigma}(t)}(D_{\pi_{\sigma}(t)} \cap D_{\pi_{\sigma}(s^*)}) = D_{\pi_{\sigma}((st)^*)}\cap D_{\pi_{\sigma}(t^*)}$ to deduce \ref{p3}. If $x \in \alpha^{-1}_{\pi_{\sigma}(t)}(D_{\pi_{\sigma}(t)} \cap D_{\pi_{\sigma}(s^*)})$, we have $x\in D_{\pi_{\sigma}(t^*)}$ and $\alpha_{\pi_{\sigma}(t)}(x)\in D_{\pi_{\sigma}(t)} \cap D_{\pi_{\sigma}(s^*)}$. Since $x\in D_{\pi_{\sigma}(t^*)}$, there exists $t_0\in\Se$ such that $(t,t_0)\in\sigma$, $x\in X_{t_0^*}$ and $\alpha_{\pi_{\sigma}(t)}(x)=\theta_{t_0}(x)$. From $\theta_{t_0}(x)=\alpha_{\pi_{\sigma}(t)}(x)\in D_{\pi_{\sigma}(s^*)}$, there exists $s_0\in\Se$ such that $(s_0,s)\in\sigma$ and $\theta_{t_0}(x)\in X_{t_0}\cap X_{s^*_0}$. As $s_0$ and $t_0$ are parallel to $s$ and $t$, respectively, we have $(s_0,t_0)\in\Se^{(2)}$. Therefore, using that $\sigma$ is a congruence and $\theta$ is an action, we obtain
$$x\in \theta_{t_0}^{-1}(X_{t_0}\cap X_{s^*_0})=X_{(s_0t_0)^*}\cap X_{t_0^*}\subseteq D_{\pi_{\sigma}((st)^*)}\cap D_{\pi_{\sigma}(t^*)}.$$

Conversely, if $x\in D_{\pi_{\sigma}((st)^*)}\cap D_{\pi_{\sigma}(t^*)}$, then there exists $r_0, t_0\in\Se$ such that $(st,r_0),(t,t_0)\in\sigma$, $x\in X_{t_0^*}\cap X_{r_0^*}$ and $\alpha_{\pi_{\sigma}(t)}(x)=\theta_{t_0}(x)$. Since $t_0$ and $r_0$ are parallel to $t$ and $st$, respectively, and $(t,(st)^*)\in\Se^{(2)}$, we must have $(t_0,r_0^*)\in\Se^{(2)}$. Using Proposition 2.7 of \cite{demeneghi-tasca-2024}, we have $\theta_{t_0}(X_{t_0^*}\cap X_{r_0^*})=X_{t_0}\cap X_{t_0r_0^*}$. Moreover, as $\sigma$ is a congruence, we have that $t_0r_0^*$ is $\sigma$-congruent to $tt^*s^*$ which in turn is $\sigma$-congruent to $s^*$. Therefore,
$$\alpha_{\pi_{\sigma}(t)}(x)=\theta_{t_0}(x)\in \theta_{t_0}(X_{t_0^*}\cap X_{r_0^*})=X_{t_0}\cap X_{t_0r_0^*} \subseteq D_{\pi_{\sigma}(t)} \cap D_{\pi_{\sigma}(s^*)},$$
which gives $x\in \alpha^{-1}_{\pi_{\sigma}(t)}(D_{\pi_{\sigma}(t)} \cap D_{\pi_{\sigma}(s^*)})$ as desired. Moreover, we have
$$\alpha_{\pi_{\sigma}(s)}(\alpha_{\pi_{\sigma(t)}}(x))=\theta_{r_0t_0^*}(\theta_{t_0}(x))=\theta_{r_0}(x)=\alpha_{\pi_{\sigma}(st)}(x),$$
and \ref{p3} is proved
\end{proof}

Finally, we establish our version of McAlister’s $P$-theorem, showing that any $E$-unitary inverse semigroupoid arises from the $P$-semigroupoid construction.

\begin{theorem}\label{P-theoremConstruct2}
Let $\Se$ is an $E$-unitary inverse semigroupoid and $\theta$ be the Munn action of $\Se$ on $E(\Se)$. If $\alpha$ is the induced action of $\Se/\sigma$ on $E(\Se)$ given in \cref{partMunn}, then $\Se \cong \Se/\sigma \ltimes_\alpha E(\Se)$.
\end{theorem}

\begin{proof}
We shall adopt the same notation used in \cref{Munn-global} and \cref{partMunn}. Notice initially that $s^*s \in X_{s^*s}=X_{s^*}\subseteq D_{\pi_{\sigma}(s^*)}$ which means that $(\pi_{\sigma}(s),s^*s)$ is a well-defined element of $\Se/\sigma \ltimes_\alpha E(\Se)$. Define $\phi\colon \Se \to \Se/\sigma \ltimes_\alpha E(\Se)$ by $\phi(s)=(\pi_{\sigma}(s),s^*s)$.
	
Let $(s,t)\in \Se^{(2)}$ and lets check that $\phi$ is a morphism of semigroupoids. Since $t^*t\in X_{t^*}\subseteq D_{\pi_{\sigma}(t)}$, we have $\alpha_{\pi_{\sigma}(t)}(t^*t)=tt^*tt^*=tt^*$. Moreover, $s^*stt^*\leq tt^*$ which gives $s^*stt^*\in X_t$ and $\alpha_{\pi_{\sigma}(t^*)}(s^*stt^*)=t^*s^*stt^*t=(st)^*(st)$. We thus have $(\phi(s),\phi(t))\in (\Se/\sigma \ltimes_\alpha E(\Se))^{(2)}$ and
\begin{align*}
\phi(s)\phi(t) &
= \big(\pi_{\sigma}(s),s^*s\big)\big(\pi_{\sigma}(t),t^*t\big)
= \Big(\pi_{\sigma}(st),\alpha_{\pi_{\sigma}(t^*)}\big(s^*s\alpha_{\pi_{\sigma}(t)}(t^*t)\big)\Big) \\ 
& = \big(\pi_{\sigma}(st),\alpha_{\pi_{\sigma}(t^*)}(s^*stt^*)\big)
= \big(\pi_{\sigma}(st),(st)^*(st)\big)
= \phi(st).
\end{align*}

It is also clear that $\phi$ is a strong morphism since $(\phi(s),\phi(t))\in (\Se/\sigma \ltimes_\alpha E(\Se))^{(2)}$ implies $(s^*s,tt^*)\in\Se^{(2)}$ which further implies that $(s,t)\in\Se^{(2)}$.

We now prove that $\phi$ is injective. If $\phi(s)=\phi(t)$ for $s,t\in\Se$, then $s$ and $t$ are $\sigma$-congruent and $s^*s=t^*t$. By \cref{lema-tecnico-idempotentes}, we have $s=ss^*s=st^*t=ts^*s=tt^*t=t$.
	
Finally, to prove that $\phi$ is surjective, let $(\pi_{\sigma}(s),e) \in \Se/\sigma\ltimes_\alpha E(\Se)$. Since $e\in D_{\pi_{\sigma}(s^*)}$, there exists $t\in\Se$ such that $(s,t)\in\sigma$ and $e\in X_{t^*}$. Noticing that $(t,e)\in S^{(2)}$, $t^*te=e$ and $(s,te)\in\sigma$ we have 
$$\phi(te)=\big(\pi_{\sigma}(te),(te)^*(te)\big)=\big(\pi_{\sigma}(te),t^*te\big)=\big(\pi_{\sigma}(te),e\big)=\big(\pi_{\sigma}(s),e\big).$$	
\end{proof}


\bibliographystyle{acm}
\bibliography{references.bib}
\end{document}